\documentclass{amsart}
\usepackage{amsmath, amsthm, amssymb, amscd, mathrsfs, eucal, epsfig}
\usepackage{pinlabel}

\begin{document}

\newtheorem{theorem}{Theorem}[section]
\newtheorem{lemma}[theorem]{Lemma}
\newtheorem{proposition}[theorem]{Proposition}
\newtheorem{corollary}[theorem]{Corollary}
\newtheorem{conjecture}[theorem]{Conjecture}
\newtheorem{question}[theorem]{Question}
\newtheorem{problem}[theorem]{Problem}
\newtheorem*{claim}{Claim}
\newtheorem*{criterion}{Criterion}
\newtheorem*{rat_thm}{Rationality Theorem}
\newtheorem*{denom_thm}{Denominator Theorem}
\newtheorem*{limit_thm}{Limit Theorem}
\newtheorem*{surgery_thm}{Surgery Theorem}

\theoremstyle{definition}
\newtheorem{definition}[theorem]{Definition}
\newtheorem{construction}[theorem]{Construction}
\newtheorem{notation}[theorem]{Notation}

\theoremstyle{remark}
\newtheorem{remark}[theorem]{Remark}
\newtheorem{example}[theorem]{Example}

\numberwithin{equation}{subsection}

\def\Z{\mathbb Z}
\def\N{\mathbb N}
\def\R{\mathbb R}
\def\Q{\mathbb Q}
\def\D{\mathcal D}
\def\E{\mathcal E}
\def\RR{\mathcal R}
\def\P{\mathcal P}
\def\F{\mathcal F}

\def\cl{\textnormal{cl}}
\def\scl{\textnormal{scl}}
\def\homeo{\textnormal{Homeo}}
\def\rot{\textnormal{rot}}
\def\area{\textnormal{area}}

\def\id{\textnormal{id}}
\def\Id{\textnormal{Id}}
\def\SL{\textnormal{SL}}
\def\GL{\textnormal{GL}}
\def\Sp{\textnormal{Sp}}
\def\PSL{\textnormal{PSL}}
\def\Hom{\textnormal{Hom}}
\def\Out{\textnormal{Out}}
\def\Aut{\textnormal{Aut}}
\def\freq{\textnormal{freq}}
\def\length{\textnormal{length}}
\def\fill{\textnormal{fill}}
\def\rank{\textnormal{rank}}
\def\til{\widetilde}
\def\supp{\textnormal{supp}}
\def\inte{\textnormal{int}}
\def\conv{\textnormal{conv}}
\def\sail{\textnormal{sail}}
\def\sign{\textnormal{sign}}
\def\CAT{\textnormal{CAT}}

\title{Scl, sails and surgery}
\author{Danny Calegari}
\address{Department of Mathematics \\ Caltech \\
Pasadena CA, 91125}
\email{dannyc@its.caltech.edu}

\date{11/1/2010, Version 0.20}

\dedicatory{This paper is dedicated to the memory of John Stallings.}
\begin{abstract}
We establish a close connection between stable commutator length in free groups and the geometry
of {\em sails} (roughly, the boundary of the convex hull of the set of integer lattice points) in
integral polyhedral cones. This connection allows us to show that the $\scl$ norm is piecewise
rational linear in free products of Abelian groups, and that it can be computed via integer programming. 
Furthermore, we show that the $\scl$
spectrum of nonabelian free groups contains elements congruent to every rational number modulo $\Z$,
and contains well-ordered sequences of values with ordinal type $\omega^\omega$.
Finally, we study families of elements $w(p)$ in free groups obtained by {\em surgery} on a fixed 
element $w$ in a free product of Abelian groups of higher rank, and show that $\scl(w(p)) \to \scl(w)$
as $p \to \infty$.
\end{abstract}

\maketitle

\section{Introduction}

Bounded cohomology, introduced by Gromov in \cite{Gromov_bounded}, proposes to {\em quantify}
homology theory, replacing groups and homomorphisms with Banach spaces and bounded linear
maps. In principle, the information contained in the bounded cohomology of a space is
incredibly rich and powerful; in practice, except in (virtually) trivial cases, this information
has proved impossible to compute. Burger-Monod \cite{Burger_Monod} wrote the following in 2000:
\begin{quote}
Although the theory of bounded cohomology has recently found many applications in various 
fields \dots \,for discrete groups it remains scarcely accessible to computation. As a matter of fact,
almost all known results assert either a complete vanishing or yield intractable
infinite dimensional spaces.
\end{quote}
Perhaps the best known exceptions are Gromov's theorem \cite{Gromov_bounded}
that the norm of the fundamental
class of a hyperbolic manifold is proportional to its volume, and Gabai's theorem \cite{Gabai}
that the Gromov norm on $H_2$ of an atoroidal $3$-manifold is equal to the Thurston norm.
Even these results only describe a finite dimensional sliver of the (typically) uncountable
dimensional bounded cohomology groups of the spaces in question.

The case of $2$-dimensional bounded cohomology is especially interesting, since it
concerns {\em extremal maps of surfaces into spaces}. In virtually every category
it is important to be able to construct and classify
surfaces of least complexity mapping to a given
target; we mention only minimal surface theory and Gromov-Witten theory as 
prominent examples. In the topological category one wants to minimize the
{\em genus} of a surface mapping to some space subject to further constraints
(e.g.\/ that the image represent a given homology class, that it be $\pi_1$-injective,
that it be a Heegaard surface, etc.) For many applications (e.g.\/ inductive arguments)
it is crucial to {\em relativize} this problem: given a space $X$ and a
(homologically trivial) loop $\gamma$ in $X$, one wants to find a surface
of least complexity (again, perhaps subject to further constraints)
mapping to $X$ in such a way as to fill $\gamma$ (i.e.\/ $\gamma$ becomes the boundary
of the surface). The problem of computing the
genus of a knot (in a $3$-manifold) is of this kind. On the algebraic side, the
relevant (bounded) homological tool to describe complexity in this context
is {\em stable commutator length}. In a group $G$, the {\em commutator length}
$\cl(g)$ of an element $g$ is the least number of commutators whose product is $g$,
and the {\em stable commutator length} $\scl(g)$ is the limit 
$\scl(g) = \lim_{n \to \infty} \cl(g^n)/n$.
Here, until very recently, the landscape was even more barren: there
were virtually no examples of groups or spaces in which stable commutator length
could be calculated exactly where it did not vanish identically (\cite{Zhuang}
is an interesting exception).

The paper \cite{Calegari_pickle} successfully showed how to compute stable
commutator length in a highly nontrivial example: that of free groups. The result
is very interesting: stable commutator length turns out to be {\em rational},
and for every rational $1$-boundary, there is a (possibly not unique) {\em best}
surface which fills it rationally, in a precise sense. 
The case of a free group is important for several reasons:
\begin{enumerate}
\item{Computing $\scl$ in free groups gives universal estimates for $\scl$ in arbitrary groups}
\item{The category of surfaces and maps between them up to homotopy is a fundamental mathematical
object; studying $\scl$ in free and surface groups gives a powerful new framework in which
to explore this category}
\item{Free groups are the simplest examples of hyperbolic groups, and are a 
model for certain other families of groups (mapping class groups, $\Out(F_n)$,
groups of symplectomorphisms) that exhibit hyperbolic behavior}
\end{enumerate}
The paper \cite{Calegari_pickle} gives an
algorithm to compute $\scl$ on elements in a free group. Refinements 
(see \cite{Calegari_scl}, Ch.~4) show how to modify this algorithm to make
it polynomial time in word length. Hence it has become possible to calculate (by computer) 
the value of $\scl$ for words of length $\sim 60$ in a free group on
two generators. The utility of this is to make it possible to perform
experiments, which reveal the existence of hitherto unsuspected phenomena in the 
$\scl$ spectrum of a free group. These phenomena suggest many new directions for research,
some of which are pursued in this paper. 

Figure~\ref{histogram_alt} is a histogram of values of $\scl$ between $1$ and $\frac 5 4$ on
alternating words of length $28$ in the (commutator subgroup of the)
free group on two generators. Here a word is {\em alternating} if the
letters alternate between one of $a^{\pm 1}$ and one of $b^{\pm 1}$.
\begin{figure}[htpb]
\labellist
\small\hair 2pt
\pinlabel $1$ at 10 -20
\pinlabel $\frac{15}{14}$ at 152.857142 -20
\pinlabel $\frac{13}{12}$ at 176.666667 -20
\pinlabel $\frac{11}{10}$ at 210 -20
\pinlabel $\frac{9}{8}$ at 260 -20
\pinlabel $\frac{8}{7}$ at 295.714286 -20
\pinlabel $\frac{7}{6}$ at 343.333333 -20
\pinlabel $\frac{17}{14}$ at 438.571428 -20
\pinlabel $\frac{19}{16}$ at 385 -20
\pinlabel $\frac{6}{5}$ at 410 -20
\pinlabel $\frac{5}{4}$ at 510 -20
\endlabellist
\centering
\includegraphics[scale=0.5]{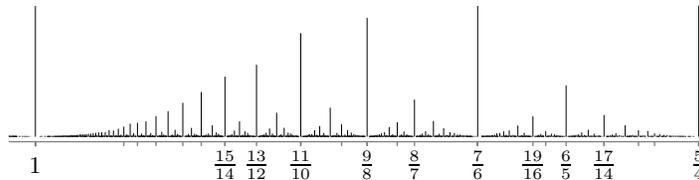}
\caption{Histogram of values of $\scl$ on alternating words of length $28$ in $F_2$.
Heights of the bars at $1$ and $\frac 5 4$ have been truncated to fit in the Figure.}\label{histogram_alt}
\end{figure}
The most salient feature of this histogram is its {\em self-similarity}. Such self-similarity
is indicative of a power law roughly of the form $\freq(p/2q) \sim q^{-\delta}$ for some $\delta$
which in this case is about $2.1$ (although there is some interesting irregularity even
in this figure, e.g.\/ the curiously ``high'' spike at $\frac {29} {28}$ 
and ``low'' spike at $\frac {39} {41}$).

Unfortunately, the algorithm developed
in \cite{Calegari_scl} is not adequate to explain the structure evident in Figure~\ref{histogram_alt}.
One reason is that this algorithm reduces the calculation of $\scl$ on a particular
element of $F_2$ to a linear programming problem, the particulars of which depend in quite
a dramatic way on the word in question. Moreover, though the algorithm is polynomial
time in word length, it is {\em not} polynomial time in ``log exponent word length'', i.e. the
notation which abbreviates a word like $aaaaaaaba^{-1}a^{-1}a^{-1}baab^{-1}b^{-1}$ to
$a^7ba^{-3}ba^2b^{-2}$. This is especially vexing in view of the fact that experiments suggest
a rich structure for the values of $\scl$ on families of words which differ only in the
values of their exponents. This is best illustrated with an example.
\begin{example}\label{simple_harmonic_example}
In $F_2$ with generators $a,b$, experiments suggest a formula
$$\scl(aba^{-1}b^{-1}ab^{-n}a^{-1}b^n) = 1 - \frac 1 {2n-2}$$ 
valid for $n\ge 2$.
\end{example}
There are two interesting aspects of this example: the fact that the values of $\scl$
(apparently) converge to a (rational) limit, and the nature of the error term (which is
a harmonic series). One of the goals of this paper is to develop a different approach
to computing $\scl$ in free groups (and some other classes of groups) which makes it
possible to rigorously verify and to explain the phenomena exhibited 
in Example~\ref{simple_harmonic_example} and more generally.

\subsection{Sails}

One surprising thing to come out of this paper is the discovery of a close connection between
stable commutator length in free groups, and the geometry of {\em sails} in integral polyhedral
cones. Given an integral polyhedral cone $V$ the {\em sail} of $V$ is the convex hull of
$\D+V$, where $\D$ is the set of integral lattice points in certain open faces of $V$ (this is
a generalization of the usual definition of a sail, in which one takes for $\D$ the set of all
integer lattice points in $V$ except for the vertex of the cone). Sails were introduced by
Klein, in his attempt to generalize to higher dimensions the theory of continued fractions. 
Given a cone $V$, define the {\em Klein function} $\kappa$ to be the function on $V$, linear
on rays, that is equal to $1$ exactly on the sail. It turns out that calculating $\scl$ on
chains in free products of Abelian groups reduces to the problem of maximizing a function $-\chi/2$
on a certain rational polyhedron (obtained by intersecting the product of two integral polyhedral cones
with a rational affine subspace). The function $-\chi/2$ is the sum of two terms, one linear, and
one which is (the restriction of) a sum of Klein functions associated to the two polyhedral cones.

This connection will be discussed further in a future paper, especially as it relates to the
statistical features evident e.g.\/ in Figure~\ref{histogram_alt}.

\subsection{Stallings}

If $A$ and $B$ are groups, one can build a $K(A*B,1)$ by wedging a $K(A,1)$ and a $K(B,1)$
along a basepoint. Given a surface $S$ and a map $f:S \to K(A*B,1)$ one can try to simplify $S$
and $f$ in two complementary ways: either by simplifying the part of $S$ that maps to
the two factors, or by simplifying the part that maps to the basepoint. In a precise sense,
the first strategy was pursued in \cite{Calegari_pickle} whereas the second strategy is
pursued in this paper. 

An interesting precursor of this latter approach is John Stallings' last paper \cite{Stallings},
which uses topological methods to factorize products of commutators in
a free product of groups into terms which are localized in the factors. It is a pleasure
to acknowledge my own great intellectual debt to John, and it seems especially serendipitous to
discover, in relatively unheralded work he did in the later part of his life, 
some beautiful new ideas which continue to inform and inspire. 

\subsection{Main results}

We now briefly describe the contents of the paper. In \S~\ref{background_section} we give
definitions, standardize notation, and recall some fundamental facts from the theory of stable
commutator length, especially the geometric definition in terms of maps of surfaces to spaces.

In \S~\ref{free_product_section} we describe a method to compute stable commutator length in
free products of Abelian groups. We show that the computation can be reduced to a kind of integer
programming problem, exhibiting a natural connection between the stable commutator length in
free groups, and the geometry of {\em sails} in integral polyhedral cones (this is explained in
more detail in the sequel). As a consequence, we derive our first main theorem:
\begin{rat_thm}
Let $G = *_i A_i$ be a free product of finitely many finitely generated free Abelian groups.
Then $\scl$ is a piecewise rational
linear function on $B_1^H(G)$. Moreover, there is an algorithm to compute $\scl$ in any
finite dimensional rational subspace.
\end{rat_thm}

In \S~\ref{surgery_section} we exploit the relationship between $\scl$ and sails developed in the
previous section, and use this to compute $\scl$ on an explicit multi-parameter infinite family of
chains in $B_1^H(F_2)$. This calculation enables us to rigorously verify the existence of certain
phenomena in the $\scl$ spectrum that were suggested by experiments. Explicitly, this calculation
allows us to prove:
\begin{denom_thm}
The image of a nonabelian free group of rank at least $4$ under $\scl$ in
$\R/\Z$ is precisely $\Q/\Z$.
\end{denom_thm}
and
\begin{limit_thm}
For each $n$, the image of the free group $F_n$ under $\scl$ contains a well-ordered sequence of values
with ordinal type $\omega^{\lfloor n/4 \rfloor}$. 
The image of $F_\infty$ under $\scl$ contains a well-ordered sequence of values with 
ordinal type $\omega^\omega$.
\end{limit_thm}

Finally, we obtain a result that explains the existence of many limit points in the $\scl$ spectrum
of free groups. Let $A*B$ and $A'*B$ be free products of free Abelian groups.
A {\em line of surgeries} is a family of surjective homomorphisms $\rho_p:A*B \to A'*B$ determined
by a linear family of surjective homomorphisms on the factors.

\begin{surgery_thm}
Fix $w \in B_1^H(A*B)$ and let $\rho_p:A*B \to A'*B$ be a line of surgeries, constant on the
second factor, and surjective on the first factor with $\rank(A') = \rank(A)-1$.
Define $w(p)=\rho_p(w)$.
Then $\lim_{p \to \infty} \scl(w(p)) = \scl(w)$.
\end{surgery_thm}

\section{Background}\label{background_section}

This section contains definitions and facts which will be used in the sequel. A basic reference
is \cite{Calegari_scl}.

\subsection{Definitions}\label{definition_subsection}

The following definition is standard; see \cite{Bavard}.

\begin{definition}
Let $G$ be a group, and $g \in [G,G]$. The {\em commutator length} of $g$, denoted
$\cl(g)$, is the smallest number of commutators in $G$ whose product is $g$. 
The {\em stable commutator length} of $g$, denoted $\scl(g)$, is the limit
$$\scl(g) = \lim_{n \to \infty} \frac {\cl(g^n)} {n}$$
\end{definition}

Commutator length and stable commutator length can be extended to finite linear sums of
groups elements as follows:
\begin{definition}\label{scl_sum_definition}
Let $G$ be a group, and $g_1,g_2,\cdots,g_m$ elements of $G$ whose product is in $[G,G]$.
Define
$$\cl(g_1 + \cdots + g_m) = \inf_{h_i \in G} \cl(g_1h_1g_2h_1^{-1}\cdots h_{m-1}g_mh_{m-1}^{-1})$$
and
$$\scl(g_1 + \cdots + g_m) = \lim_{n \to \infty} \frac {\cl(g_1^n + \cdots + g_m^n)} {n}$$
\end{definition}
\begin{remark}
Note that $\scl$ will be finite if and only if the product of the $g_i$ is trivial in $H_1(G;\Q)$.
\end{remark}

The function $\scl$ may be interpreted geometrically as follows.

\begin{definition}
Let $S$ be a compact surface with components $S_i$. Define
$$\chi^-(S) = \sum_i \min(0,\chi(S_i))$$
where $\chi$ denotes Euler characteristic.
\end{definition}

In words, $\chi^-(S)$ is the sum of the Euler characteristics over all components of $S$
for which $\chi$ is non-positive.

\begin{definition}\label{admissible_surface}
Let $g_1,g_2 \cdots g_m \in G$ be given so that the product of the $g_i$ is trivial in
$H_1(G;\Q)$.
Let $X$ be a space with $\pi_1(X) = G$. For each $i$, Let $\gamma_i:S^1 \to X$ represent the
conjugacy class of $g_i$ in $G$. 

Suppose $S$ be a compact oriented surface. A map $f:S \to X$ for which there is a diagram
$$ \begin{CD}
\partial S @>i>> S \\
@V\partial fVV  @VVfV \\
\coprod_i S^1 @>\coprod_i\gamma_i>> X 
\end{CD}$$
where $i:\partial S \to S$ is the inclusion map, and $\partial f_*[\partial S] = n[\coprod_i S^1]$ in $H_1$
for some integer $n(S) \ge 0$, is said to be {\em admissible}, of degree $n(S)$.
\end{definition}

\begin{remark}
The sign of $n(S)$ is changed by orienting $S$ oppositely.
\end{remark}

The geometric definition of $\scl$ asks to minimize the ratio of $-\chi^-$ to degree over all
admissible surfaces.

\begin{proposition}\label{geometric_formula_for_scl}
With notation as above, there is an equality
$$\scl_G(g_1 + \cdots + g_m) = \inf_{S} \frac {-\chi^-(S)} {2n(S)}$$
over all admissible compact oriented surfaces $S$.
\end{proposition}
See \cite{Calegari_scl}, Prop.~2.68 for a proof.

\medskip

If $f:S \to X$ is admissible, then $S$ and therefore $\partial S$ are oriented.
Some components of $\partial S$ might map by $\partial f$ to $\coprod_i S^1$ with zero
or even negative degree. Boundary components mapping with opposite degree to the same
circle can be glued up after passing to a suitable cover. Hence the following proposition can
be proved.

\begin{proposition}\label{admissible_maps_oriented}
If $f:S \to X$ is admissible, there is $f':S' \to X$ admissible such that
$\partial f':\partial S' \to \coprod_i S^1$ has positive degree on every component, and
$-\chi^-(S')/2n(S') \le -\chi^-(S)/2n(S)$.
\end{proposition}
See \cite{Calegari_scl}, Cor.~4.29 for a proof. An admissible surface with the property
discussed above is said to be {\em positive}. In the sequel all the admissible surfaces
we discuss will be positive, even if we do not explicitly say so.

\medskip

\begin{definition}
An admissible surface $S$ realizing the infimum for $g_1,\cdots,g_m$ (i.e. for which $\scl(g_1 + \cdots + g_m) = -\chi^-(S)/2n(S)$)
is said to be {\em extremal}.
\end{definition}

\begin{remark}
Extremal surfaces are $\pi_1$-injective, and have other useful properties.
\end{remark}

Given a group $G$, let $(C_*(G;\R),\partial)$ denote the complex of real group chains,
whose homology is the real (group) homology of $G$ (see Mac Lane \cite{Maclane}, Ch.~IV, \S~5). 
Let $B_n(G;\R)$ denote the subspace of real group $n$-boundaries. By
Definition~\ref{scl_sum_definition}, we can think of $\scl$ as a function on the set of
integral group $1$-boundaries. This function is linear on rays and subadditive, and therefore
admits a unique continuous linear extension to $B_1(G)$.

Let $H(G)$ (for {\em homogeneous}) denote the subspace of $B_1(G)$
spanned by chains of the form $g^n - ng$ and $g - hgh^{-1}$
for all $g,h \in G$ and $n \in \Z$. Then $\scl$ vanishes on $H$ and descends to a pseudo-norm on
$B_1(G)/H(G)$. For general $G$ this pseudo-norm is not a true norm, but in many cases
of interest (e.g.\/ for fundamental groups of hyperbolic manifolds), $\scl$ is a genuine norm on
$B_1(G)/H(G)$. See \cite{Calegari_scl}, \S~2.6 for proofs
of these basic facts. We usually denote this quotient space by $B_1^H(G)$ or just $B_1^H$ if $G$ is
understood.

\section{Free products of Abelian groups}\label{free_product_section}

The purpose of this section is to prove that $\scl$ is piecewise rational linear in
$B_1^H(G)$, where $G$ is a free product of Abelian groups. Along the way we develop some
additional structure which is important for what follows.

\subsection{Euler characteristic with corners}

We will obtain surfaces by gluing up simpler surfaces along segments in their
boundary. Since ordinary Euler characteristic is not additive under such gluing,
we consider surfaces with {\em corners}. Technically, a corner should be
thought of as an orbifold point with angle $\pi/2$ (in contrast to a ``smooth''
boundary point where the angle is $\pi$).

When two surfaces with boundary are glued along a pair of segments in their
boundary, the interior points of the segments should be smooth, and the endpoints
should be corners. In the glued up surface, the boundary points which result
from identifying two corners should be smooth.

If $S$ is a surface as above, let $c(S)$ denote the number of corners. Define
the {\em orbifold Euler characteristic} of $S$, denoted $\chi_o(S)$, by the formula
$$\chi_o(S) = \chi(S) - \frac {c(S)} 4$$

With this convention, $\chi_o$ is additive under gluing, and
$\chi_o = \chi$ for a surface with no corners.
In the sequel we will only consider surfaces with an even number of corners, 
so $\chi_o$ will always be in $\frac 1 2 \Z$.

\subsection{Decomposing surfaces}\label{constructing_extremal_surfaces_subsection}

Throughout the remainder of this section, we fix a group $G=A*B$ where 
$A$ and $B$ are free Abelian groups, and a finite set
$Z$ of (nontrivial) conjugacy classes in $G$. We are interested in the restriction of $\scl$
to the space $\langle Z \rangle \cap B_1^H(G)$ of homologically trivial chains with support in $Z$. 

Let $K_A$ and $K_B$
be a $K(A,1)$ and a $K(B,1)$ respectively (for instance, we could take $K(\cdot,A)$ to be a torus
of dimension equal to the rank of $A$, and similarly for $B$) and let $K = K_A \vee K_B$ be a $K(G,1)$.
We denote the base point of $K$ by $* = K_A \cap K_B$.

A homotopically essential map $\gamma:S^1 \to K$ is {\em tight} if it has one of the following
two forms:
\begin{enumerate}
\item{the image of $\gamma$ is a loop contained entirely in $K_A$ or in $K_B$ (we call such maps
{\em Abelian loops}); or}
\item{the circle $S^1$ can be decomposed into intervals, each of which is taken alternately
to an essential based loop in one of $K_A,K_B$.}
\end{enumerate}
Every free homotopy class of map to $K$ has a tight representative. 

A (nonabelian) tight loop $\gamma:S^1 \to K$ induces a polygonal structure on $S^1$, with one edge
for each component of the preimage of $K_A$ or $K_B$ and vertex set $\gamma^{-1}(*)$.
By convention, we also introduce a polygonal structure on $S^1$ when $\gamma$ is an Abelian
loop, with one (arbitrary) vertex and one edge.

For each element of $Z$, choose a tight loop in the correct conjugacy class.
The union of these tight loops can be thought of as an oriented $1$-manifold $L$
(with one component for each element of $Z$) together with a map $\Gamma:L \to K$.
As above, $\Gamma$ induces a polygonal structure on $L$. Each oriented edge in this polygonal
structure is mapped either to $K_A$ or to $K_B$. Let $T(A)$ denote the set of
{\em $A$-edges} and $T(B)$ the set of {\em $B$-edges}. Note that each $A$ or $B$ edge
can be thought of as a homotopy class of loop in $K_A$ or $K_B$, and therefore determines
an element of the (fundamental) group $A$ or $B$.

\medskip

Let $f:S \to X$ be an admissible surface. After a homotopy, we assume (in the notation
of Proposition~\ref{geometric_formula_for_scl}) that 
$\partial f:\partial S \to L$ is a
covering map, and that $f$ is transverse to $*$ (i.e.\/ $f^{-1}(*)$ is a system of proper
arcs and loops). Furthermore, we assume 
(by Proposition~\ref{admissible_maps_oriented}) that $\partial f:\partial S \to L$ is
{\em orientation-preserving}.

Denote by $F$ the preimage $f^{-1}(*)$ in $S$; by hypothesis, $F$ is a system of proper arcs and loops in $S$.
In anticipation of what is to come, we refer to the components of $F$ as {\em $\sigma$-edges}.
Since $f$ maps $F$ to $*$, loops of $F$ can be eliminated by compression (innermost first).
Since $f$ restricted to $\partial S$ is a covering map, every arc of $F$ is essential. Thus without
loss of generality we assume $F$ is a system of essential proper arcs in $S$. Cut $S$ along
$F$ and take the path closure to obtain two surfaces $S_A$ and $S_B$, which
are the preimages under $f$ of $K_A$ and $K_B$ respectively, 
and satisfy $S_A \cap S_B = F$.

Each component of $\partial S_A$ either maps entirely to $K_A$ (those which cover Abelian loops)
or decomposes into arcs which alternate between components of $F$ and arcs
which map to elements of $T(A)$; we refer to the second kind of arcs as {\em $\tau$-edges}. 
In order to treat everything uniformly, we blow up
the vertices on Abelian loops into intervals, which we refer to as {\em dummy} 
$\sigma$-edges. A $\sigma$-edge which is not a dummy edge is {\em genuine}.
Thus $\partial S_A$ can be thought of as a union of
polygonal circles, whose edges alternate between $\sigma$-edges and $\tau$-edges.

The surface $S_A$ and $S_B$ naturally have the structure of surfaces with corners precisely
at points of $F \cap \partial S$. In particular,
$$\chi_o(S_A) = \chi(S_A) - \frac 1 2 \text{ number of components of } F$$
The number of components of $F$ is equal to the number of genuine $\sigma$-edges.
on $S_A$ (which is therefore equal to the number of genuine $\sigma$-edges on $S_B$).
Since $S$ has no corners,
$$\chi(S) = \chi_o(S) = \chi_o(S_A) + \chi_o(S_B) = \chi(S_A) + \chi(S_B) - \text{number of components of } F$$

\subsection{Encoding surfaces as vectors}

We would like to reduce the computation of $\scl$ to a finite dimensional linear programming
problem. The main difficulty is that it is difficult to find a useful parameterization of the
set of all admissible surfaces. However, in the end, all we need to know about an admissible surface
is $-\chi^-$ and degree.

We need to keep track of two different kinds of information: the number and kind of $\tau$-edges which appear
in $\partial S_A$, and the number and kind of $\sigma$-edges which appear. Each
oriented $\sigma$ edge runs from the end of one oriented $\tau$-edge to the start of
another oriented $\tau$-edge, and can therefore be encoded as an ordered pair of $\tau$-edges; 
i.e. as an element of $T(A) \times T(A)$. 
However, not {\em every} element of $T(A) \times T(A)$ can arise in this way:
the only $\sigma$-edges associated to Abelian loops are the ``dummy'' $\sigma$-edges.
Consequently we let $T_2(A)$ denote the set
of ordered pairs $(\tau,\tau')$ with $\tau,\tau' \in T(A)$ subject to the constraint
that if either of $\tau,\tau'$ is an Abelian loop, then $\tau = \tau'$.

Let $C_1(A)$ denote the $\R$-vector space with basis $T(A)$, and $C_2(A)$ the $\R$-vector
space with basis $T_2(A)$.
The oriented surface $S_A$ determines a set of oriented $\sigma$-edges and therefore
a non-negative integral vector $v(S_A) \in C_2(A)$. This vector is not arbitrary
however, but is subject to two further linear constraints which we now describe.

Define a linear map $\partial: C_2(A) \to C_1(A)$ on basis vectors by $\partial(\tau,\tau') = \tau - \tau'$,
and extend by linearity. Since each $\tau$-edge is contained between exactly two $\sigma$ edges, 
$\partial \circ v(S_A) = 0$.
Similarly define $h:C_2(A) \to A \otimes \R$ by $h(\tau,\tau') = \frac 1 2 (\tau + \tau')$ and
extend by linearity. By definition, $h\circ v(S_A)$ is equal to the image of $[\partial S_A]$ in $H_1(K_A) = A$. 
Since $\partial S_A$ is a boundary, $h \circ v(S_A) = 0$.

\begin{definition}
Let $V_A$ be the convex rational cone of non-negative vectors
$v$ in $C_2(A)$ satisfying $\partial(v) = 0$ and $h(v)=0$.
\end{definition}

Note that $V_A$ is the cone on a compact convex rational polyhedron. A surface $S_A$ as above
determines an integral vector $v(S_A) \in V$. 
Conversely we will see that for every non-negative integral vector $v \in V_A$ there are many possible
surfaces $S_A$ with $v(S_A) = v$. For such a $S_A$, the number of genuine $\sigma$-edges depends only
on the vector $v$. However, it is important to be able to choose such a surface $S_A$
with $\chi(S_A)$ as big as possible. Finding such an $S_A$ is an interesting combinatorial problem, which
we now address.

\begin{definition}
A {\em weighted directed graph} is a directed graph $\Sigma$ together with an assignment
of a non-negative integer to each edge of $\Sigma$. The {\em support} of a weighted directed
graph is the subgraph of $\Sigma$ consisting of edges with positive weights, together with
their vertices.
\end{definition}

Let $v \in V_A$ be integral. Define a weighted directed graph $X(v)$ as follows. First let $\Sigma$
denote the directed graph with vertex set $T(A)$ and edge set $T_2(A)$. Edges of $\Sigma$ 
correspond to basis vectors of $C_2(A)$. Give each edge a weight equal to the coefficient of $v$ when
expressed in terms of the natural basis. 

\begin{definition}
Given a graph $X(v)$ as above, let $\supp(v)$ denote the support of $X(v)$
(so that $\supp(v)$ is a subgraph of $\Sigma$), and let $|X(v)|$ denote the number of components
of $\supp(v)$.
\end{definition}

\begin{lemma}
Let $v \in V$ be a non-negative integral vector. Then there is a planar surface $S_A$ with $v(S_A) = v$ 
and with $|X(v)|$ boundary components. Moreover for any surface $S_A$ with $v(S_A)=v$, the number of boundary
components of $S_A$ is at least $|X(v)|$.
\end{lemma}
\begin{proof}
We construct a component of $\partial S_A$ for each component of $X(v)$. Since $\partial(v)=0$,
the indegree (i.e. the sum of the weights on the incoming edges) and the outdegree (i.e.
the sum of the weights on the outgoing edges) at each vertex of $X(v)$ are equal. The same
is true for each connected component of $X(v)$. A connected directed weighted graph with equal
indegree and outdegree at each vertex admits an {\em Eulerian circuit}; i.e. a directed
circuit which passes over each edge a number of times equal to its weight. This fact is
classical; see e.g.\/ \cite{Bollobas}, \S~I.3. The vertices visited
in such a circuit (in order) determine a sequence of elements of $T(A)$. Together with one
$\sigma$ edge (mapping to $*$) between each pair in the sequence, we construct a circle
and a map to $K_A$. If we do this for each component of $X(v)$, we obtain a $1$-manifold $D$
and a map to $K_A$. The image of $D$ in $H_1(K_A)=A$ is equal to $h(v) = 0$, so $D$ bounds a
map of a surface $S_A'$ to $K_A$. Since $A$ is Abelian, every embedded once-punctured torus
in $S_A'$ has boundary which maps to a homotopically trivial loop in $K_A$, and can therefore be
compressed. After finitely many such compressions, we obtain a planar surface $S_A$ as claimed.

Conversely, if $S_A$ is a surface with $v(S_A)=v$ then every boundary component determines an Eulerian
circuit in $X(v)$ in such a way that the sum of the degrees of these circuits is equal to the weight.
In particular, each component of $X(v)$ is in the image of at least one boundary component, and
the Lemma is proved.
\end{proof}

A vector $v$ in $C_2(A)$ can be thought of as a finite linear combination of elements of
$T_2(A)$. Define $|v|$ to be the sum of the coefficients of $v$ {\em excluding the coefficients
corresponding to Abelian loops}. Hence for $v=|v(S_A)|$, the number $|v|$ is just the number of
genuine $\sigma$-edges in $\partial S_A$. We conclude that
$$\chi_o(S_A) = \chi(S_A) - \frac 1 2 |v|$$
In order to determine $\scl$ we would like to construct surfaces $S$ with a given $v(S)$
with $\chi(S)$ as large as possible. The first easy, but key observation is the following:

\begin{lemma}\label{negative_is_zero}
For any $v\in V_A$ and any positive integer $n$, the graphs $X(v)$ and $X(nv)$ 
have the same number of components.
\end{lemma}
\begin{proof}
The graphs are the same, but the weights are scaled by $n$.
\end{proof}

It follows that for any $v$ and any $\epsilon > 0$
one can find a surface $S_A$ with $v(S_A) = nv$ and $|\chi(S_A)|/n < \epsilon$. Hence we may
take $\chi_o(S_A)$ to be projectively as close to $-\frac 1 2|v|$ as we like. As far as surfaces
with $\chi(S_A)\le 0$ are concerned, this is the end of the story. However it is very important
to control the complexity of surfaces $S_A$ with $v(S_A)=v$ which contain disk components, and it
is this which we focus on in the next section. 

\subsection{Disk vectors and sails}

\begin{definition}
A non-negative nonzero integral vector $v \in V_A$ with $\supp(v)$ connected is called a
{\em disk vector}.
\end{definition}

Notice that a disk vector $v$ can contain no Abelian loops. That is, if $e$ in $T_2(A)$
is of the form $e = (\tau,\tau)$ where $\tau \in T(A)$ is an Abelian loop, then
the coefficient of $e$ in the vector $v$ is necessarily zero. For, the hypothesis that $v$ is a disk vector
implies that $e$ is the {\em only} nonzero coefficient in $v$. But this implies that
$h(v)$ is a nonzero multiple of $h(\tau) \in A$. Since $A$ is free and $h(\tau)$ is nonzero,
$h(v)$ is nonzero, contrary to the hypothesis that $v \in V_A$. This proves the claim.

In particular, for $v$ a disk vector, $|v|$ is the ordinary $L^1$ norm of
the vector $v$, and is therefore a good measure of its complexity.

\begin{definition}\label{admissible_expression}
Let $v$ be a (not necessarily integral) vector in $V_A$. An expression of the form
$$v = \sum t_i v_i + v'$$
is {\em admissible} if each $v_i$ is a disk vector (and, in particular, is integral), 
each $t_i$ is positive, and $v' \in V_A$.
\end{definition}

We are now in a position to define a suitable function $\chi_o$ on $V_A$.

\begin{definition}\label{klein_function}
Define $\chi_o$ on $V_A$ by
$$\chi_o(v) = \sup \sum_i t_i - \frac 1 2 |v|$$
where the supremum is taken over all admissible expressions $v = \sum t_i v_i + v'$;
i.e. expressions where $v' \in V_A$, the $t_i > 0$, and each $v_i$ is a disk vector.
\end{definition}

\begin{lemma}
For any surface $S_A$, there is an inequality $\chi_o(v(S_A)) \ge \chi_o(S_A)$. Conversely,
for any rational vector $v \in V_A$ and any $\epsilon > 0$ there is an integer $n$ and
a surface $S_A$ with $v(S_A) = nv$ such that $|\chi_o(S_A)/n - \chi_o(v)| \le \epsilon$.
\end{lemma}
\begin{proof}
Let $S_A$ be a surface, with disk components $D_1,\cdots,D_m$ and $S_A' = S_A - \cup_i D_i$.
Corresponding to this there is an admissible expression
$$v(S_A) = \sum v(D_i) + v(S_A')$$
Now, $\chi_o(S_A') = \chi(S_A')-|v(S_A')|/2$, and since $S_A'$ contains no disk components,
$\chi(S_A')\le 0$. Moreover, $\chi_o(\cup_i D_i) = \sum_i (1-|v(D_i)|/2)$. Hence
$$\chi_o(S_A) = \sum_i(1-|v(D_i)|/2) + \chi(S_A') - |v(S_A')|/2 \le \sum_i 1 - |v(S_A)|/2 \le \chi_o(v(S_A))$$
proving the first claim.

The idea behind the proof of the second claim is as follows. 
Since $\chi_o(S_A) = \chi(S_A) - |v|/2$, to maximize $\chi_o(S_A)$ for a given $v(S_A)$ 
is to maximize $\chi(S_A)$. Since components with $\chi\le 0$ can be projectively replaced 
by components with $\chi$ as close to $0$ as desired (by Lemma~\ref{negative_is_zero}), the
goal is to (projectively) maximize the number of disks used. In more detail:
if $v$ is rational, and $v = \sum t_i v_i + v'$ is an admissible expression, we
can find another admissible expression $v = \sum t_i' v_i + v''$ where the $t_i'$ and $v''$
are rational, and $\sum |t_i - t_i'| \le \epsilon/2$ for any fixed positive $\epsilon$. After multiplying
through by a big integer $n$ to clear denominators, we can find a surface $S_A$ with
$v(S_A) = nv$, with $\sum_i nt_i'$ disk components. Now, it may be that the non disk components
have $\chi$ negative, but by Lemma~\ref{negative_is_zero} we can projectively replace this
part of the surface by a planar surface with $\chi$ very close to $0$. Hence
after possibly replacing $n$ by a much bigger integer, we can find $S_A$ with $v(S_A)=nv$
such that $|\chi_o(S_A)/n - (\sum t_i' - \frac 1 2 |v|)| \le \epsilon/2$.
This completes the proof.
\end{proof}

It remains to study the function $\chi_o$.
Equivalently, we study 
$$\kappa(v) = \sup \sum_i t_i = \chi_o(v) + |v|/2$$ 
which we call the {\em Klein function} of $V_A$. Let $\D_A \subset V_A$ denote the set of disk vectors,
and let $\D_A + V_A$ denote {\em Minkowski sum} of $\D_A$ and $V_A$; i.e.
the set of vectors of the form $d+v$ for $d \in \D_A$ and $v \in V_A$. 
Let $\conv(\cdot)$ be the function which assigns to a subset of
a linear space its convex hull. Taking convex hulls commutes with Minkowski sum.
Note that since $V_A$ is convex, $\conv(\D_A + V_A) = \conv(\D_A) + V_A$.

\begin{lemma}\label{klein_is_boundary}
The Klein function $\kappa$ is a non-negative concave linear function on $V_A$.
The subset of $V_A$ on which $\kappa=1$ is the boundary of $\conv(\D_A+V_A)$.
\end{lemma}
\begin{proof}
If $v_1,v_2$ are elements of $V_A$, the sum of admissible expressions for $v_1$ and $v_2$
is an admissible expression for their sum. This proves concavity. Non-negativity is obvious from the
definition. To prove the last assertion, note that an admissible expression
$v = \sum t_i v_i + v'$ exhibits $v/(\sum t_i)$ as an element of $\conv(\D_A) + V_A = \conv(\D_A + V_A)$.
\end{proof}

The following lemma, while elementary, is crucial.

\begin{lemma}\label{boundary_is_finite_polyhedron}
The sets $\conv(\D_A)$ and $\conv(\D_A + V_A)$ are finite sided convex closed polyhedra, 
whose vertices are elements of $\D_A$.
\end{lemma}
\begin{proof}
For each open face $F$ of the polyhedron $V_A$ (of any codimension $\ge 0$), the support $\supp(v)$
is constant on $F$. We denote this common support by $\supp(F)$. By definition,
$\D_A$ is the union of the integer lattice points in those open faces $F$ of $V_A$ for which
$\supp(F)$ is connected. 

If $F$ is an open polyhedral cone, the convex hull of the set of integer lattice points
in $F$ is classically called a {\em Klein polyhedron}, and its boundary is called a {\em sail}.
It is a classical fact, which goes back at least to Gordan \cite{Gordan} that if $F$ is
rational, the set of lattice points in the closure of $F$ has a finite basis (as an additive semigroup)
which is sometimes called a {\em Hilbert basis}, and the set of lattice points in the interior
is a finitely generated module over this semigroup (see e.g.\/ Barvinok \cite{Barvinok}). 
Consequently the Klein polyhedron 
is finite sided, and its vertices are a subset of a module basis. Hence $\conv(F \cap \D_A)$ is a finite
sided closed convex polyhedron for each $F$. Since $V_A$ has only finitely many faces,
the same is true of $\conv(\D_A)$ and therefore also for $\conv(\D_A + V_A)$.
The vertices of each $\conv(F \cap \D_A)$ are in $\D_A$, so the same is true for $\conv(\D_A)$
and $\conv(\D_A + V_A)$.
\end{proof}

Consequently, from Lemma~\ref{klein_is_boundary} and Lemma~\ref{boundary_is_finite_polyhedron} 
we make the following deduction:

\begin{lemma}\label{norm_piecewise_linear}
The function $\chi_o$ on $V_A$ is equal to the minimum of a finite set of
rational linear functions.
\end{lemma}
\begin{proof}
This is true for $\kappa$, and therefore for $\kappa - |v|/2$.
\end{proof}

\begin{remark}
There is a close connection between vertices of the Klein polyhedron and continued fractions. 
If $F$ is a sector in $\R^2$, the sail is topologically a copy
of $\R$, and the vertices of the sail are integer lattice points in $\Z^2$ whose ratios are the
continued fraction approximations to the slopes of the sides of $\overline{F}$. Klein \cite{Klein}
introduced Klein polyhedra and sails (for not necessarily rational polyhedral cones $F$) 
in an effort to generalize the theory of continued fractions to higher dimensions.
In recent times this effort has been pursued by Arnold \cite{Arnold} and his school.
\end{remark}

\subsection{Rationality of scl}

We are now in a position to prove the main theorem of this section.

\begin{theorem}[Rationality]\label{rationality_theorem}
Let $G = *_i A_i$ be a free product of finitely many finitely generated free Abelian groups. 
Then $\scl$ is a piecewise rational
linear function on $B_1^H(G)$. Moreover, there is an algorithm to compute $\scl$ in any
finite dimensional rational subspace.
\end{theorem}
\begin{proof}
We first prove the theorem in the case $G=A*B$ where $A$ and $B$ are free and finitely generated
as above.

We have rational polyhedral cones $V_A,V_B$ in $C_2(A)$ and $C_2(B)$ respectively,
which come together with convex piecewise rational linear functions $\chi_o$. 
There is a rational subcone $Y \subset V_A \times V_B$ consisting of pairs of vectors $(v_A,v_B)$
in $V_A \times V_B$ which can be
{\em glued up} in the following sense. For each co-ordinate $(\tau_A,\tau_A') \in T_2(A)$
whose entries are not Abelian loops, there is a corresponding co-ordinate $(\tau_B,\tau_B') \in T_2(B)$
determined by the property that as oriented arcs in $L$, the arc $\tau_B$ is followed by $\tau_A'$,
and $\tau_A$ is followed by $\tau_B'$. Say that this pair of co-ordinates are {\em paired}.
Then $Y$ is the subspace consisting of pairs of vectors whose paired co-ordinates are equal.
Define $\chi$ on $Y$ by $\chi(v_A,v_B) = \chi_o(v_A) + \chi_o(v_B)$. 
By Lemma~\ref{norm_piecewise_linear}, the function $\chi$ is equal to the minimum of a finite set
of rational linear functions on $Y$.
Finally, there is a linear map $d:Y \to H_1(L)$ with the property that a surface $S = S_A \cup S_B$
with $(v(S_A),v(S_B)) = y \in Y$ satisfies $\partial f_*(\partial S) = dy \in H_1(L)$.

Define a linear programming problem as follows. For $l \in H_1(L)$, define $Y_l \subset Y$
to be the polyhedron which is equal to the preimage $d^{-1}(l)$. Then define
$$\scl(l) = -\max_{y \in Y_l} \chi(y)/2$$
Since $Y$ and therefore $Y_l$ are finite sided rational polyhedra, and $\chi$ is the minimum
of a finite set of rational linear functions on these polyhedra, the maximum of $\chi$ on $Y_l$
can be found algorithmically by linear programming
(e.g.\/ by Dantzig's simplex method \cite{Dantzig}), 
and is achieved precisely on a rational subpolyhedron of $Y_l$. Note that although
maximizing the minimum of several linear functions is ostensibly a nonlinear optimization
problem, it may be linearized in the standard way, by introducing extra slack variables, and
turning the linear terms (over which one is minimizing) into {\em constraints}. See e.g.\/
Dantzig \cite{Dantzig} \S~13.2.

\medskip

The case of finitely many terms is not substantially more difficult. There is a cone
$V_i$ and a piecewise rational linear function $\chi_o$ for each $A_i$, and a
slightly more complicated gluing condition to define the subspace $Y$, but there are no
essentially new ideas involved. One minor observation is that one should build a $K(G,1)$
by gluing up $K(A_i,1)$'s so that no three factors are attached at the same point. This
ensures that the surfaces $S_i$ mapping to each factor are glued up along 
genuine $\sigma$-arcs in pairs, and not in more complicated combinatorial configurations.
We leave details to the reader.
\end{proof}

\begin{remark}
If $G = *_i A_i$ where each $A_i$ is finitely generated Abelian but not necessarily
torsion free, there is a finite index subgroup $G'$ of $G$ which is a free product of
free Abelian groups. The piecewise rational linear property of $\scl$ is inherited by
finite-index supergroups. Hence $\scl$ is piecewise rational linear on $B_1^H(G)$
in this case too. A similar observation applies to amalgamations of such groups 
over finite subgroups.
\end{remark}

A perhaps surprising corollary of the method of proof is the following:

\begin{corollary}\label{injective_is_isometry}
Let $\lbrace A_i \rbrace$ and $\lbrace B_i \rbrace$ be finite families of finitely generated free
Abelian groups. For each $i$, let $\rho_i: A_i \to B_i$ be an injective homomorphism, and let 
$\rho: *_i A_i \to *_i B_i$
be the corresponding injective homomorphism. Then $\rho$ induces an isometry of the
$\scl$ norm. That is, for all chains $c \in B_1^H(*_i A_i)$, there is equality
$\scl(c) = \scl(\rho(c))$.
\end{corollary}
\begin{proof}
An injective homomorphism $\rho_i:A_i \to B_i$ induces an injective homomorphism of vector spaces
$A_i \otimes \R \to B_i \otimes \R$.
The only place in the calculation of $\scl$ that the groups $A_i$ enter is in the homomorphisms
$h:C_2(A_i) \to A_i \otimes \R$, and the map $h$ is only introduced in order to determine
the subspace $h^{-1}(0)$. Since $(\rho_i \circ h)^{-1}(0) = h^{-1}(0)$, the linear programming
problems defined by chains $c$ and $\rho(c)$ are the same, so the values of $\scl$ are the same.
\end{proof}

\begin{example}\label{power_isometry_example}
Corollary~\ref{injective_is_isometry} is interesting even (especially?) in the
rank $1$ case. Let $G=F_2$, freely generated by elements $a,b$. Then for any non-zero
integers $n,m$ the homomorphism $\rho:F_2 \to F_2$ defined by $\rho(a) = a^m, \rho(b) = b^n$ is
an isometry for $\scl$. Hence (for instance) every value of $\scl$ which is achieved in
a free group is achieved on infinitely many automorphism orbits of elements.

The composition of an arbitrary alternating sequence of automorphisms and injective 
homomorphisms as above can be quite complicated, and shows that $B_1^H(F_2)$ admits a
surprisingly large family of (not necessarily surjective) isometries.
\end{example}

If $G$ is (virtually) free, every vector in $B_1^H(G)$ with positive $\scl$ norm
rationally bounds an extremal surface, by the main theorem of \cite{Calegari_pickle}. 
However, if $G$ is a free product of Abelian groups
of higher rank, extremal surfaces are {\em not} guaranteed to exist. For a vector
$v \in V_A$ to be represented by an injective surface it is necessary that it
should be expressible as a sum $v = \sum v_i$ where each $v_i$ is in $V_A$, and $|X(v_i)| \le 2$
for each $i$. The $v_i$ correspond to the connected components of $S_A$ with $v(S_A) = v$.
Since $A$ is Abelian, for $\pi_1(S) \to A$ to be injective, every component of $S_A$ must
be either a disk (in which case $|X(v_i)| = 1$) or an annulus (in which case $|X(v_i)| = 2$).

\begin{example}\label{no_injective_representative}
In $\Z * \Z^2$, let the $\Z$ factor be generated by $a$, 
and let $v_1,v_2$ be generators for the $\Z^2$ factor.
The chain $c=av_1^2a^{-1}v_1^{-1} + v_2 + v_1^{-1}v_2^{-1}$ satisfies $\scl(c) = 1/2$, but no
extremal surface rationally fills $c$, and in fact, there does not even exist a $\pi_1$-injective
surface filling a multiple of $c$. To see this, observe that every non-negative
$v \in V_B$ has $|X(v)|\ge 3$, and therefore every surface $S_B$ with $v(S_B) = v$ 
has nonabelian (and therefore non-injective) fundamental group. 

Let $G$ be the group obtained by doubling $\Z * \Z^2$ along $c$. Notice that $G$ is
$\CAT(0)$, since a $K(G,1)$ can be obtained by attaching three flat annuli to two copies of
$S^1 \vee T^2$ along pairs of geodesic loops corresponding to the terms in $c$.
The Gromov norm on $H_2(G;\Q)$ is piecewise rational linear. On the other hand, if
$\alpha \in H_2(G;\Q)$ is any nonzero class obtained by gluing relative classes on either side
along $c$, then no surface mapping to a $K(G,1)$ in the projective class of $\alpha$
can be $\pi_1$-injective.
\end{example}

\begin{remark}
Example~\ref{no_injective_representative} suggests a connection to the simple loop conjecture
in $3$-manifold topology.
\end{remark}

\begin{example}\label{monotone_words}
The support of a disk vector cannot include a vertex corresponding to an Abelian loop. This observation
considerably simplifies the calculation of $\scl$ on certain chains. Consider
a chain of the form $w = a^{-\alpha} + b^{-\beta} + w'$ where $\alpha$ and $\beta$ are positive, and
$w'$ is either a single word or a chain composed only of the letters $a$ and $b$ (and not their
inverses). Suppose further that $w \in B_1^H(F_2)$, so that $\scl(w)$ is finite. Then by the remark
above, there are no disk vectors, so $\chi_o = -|v|/2$ and $\scl(w) \in \frac 1 2 \Z$.

Explicitly, suppose
$w = a^{-\alpha} + b^{-\beta} + \sum w_i$ where each $w_i$ is of the form
$$w_i = a^{\alpha_{i,1}}b^{\beta_{i,1}}\cdots a^{\alpha_{i,n_i}}b^{\beta_{i,n_i}}$$ 
where each $\alpha_{i,j}$ and each $\beta_{i,j}$ is positive, and $\sum_{i,j} \alpha_{i,j} = \alpha$,
$\sum_{i,j} \beta_{i,j} = \beta$. Recall that
Abelian loops do not contribute to $|v_A|$ or $|v_B|$. If $S$ is a surface with $v(S) = (Nv_A,Nv_B)$
then $\partial S$ wraps around each $\tau$-edge with multiplicity exactly $N$. Hence each $w_i$ 
contributes $n_i$ to $v_A$ and similarly for $v_B$, and therefore $|v_A|=|v_B|=\sum_i n_i$.
In particular, $\chi$ is constant on the polyhedron $Y_l$, and $\scl(w)=\frac 1 2 \sum_i n_i$.
\end{example}

In fact, the same argument shows that $|v_A|$ and $|v_B|$ are constant on $Y_l$, and therefore we can
calculate $\scl$ by maximizing $\kappa$ instead of $\chi_o$. We record this fact as a proposition:

\begin{proposition}\label{norm_is_linear}
For any chain $w$ and any $l \in H_1(L;\Z)$, the functions $|v_A|$ and $|v_B|$ are constant on $Y_l$,
and take values in $\Z$, with notation as above.
\end{proposition}

\section{Surgery}\label{surgery_section}

In this section we study how $\scl$ varies in {\em families} of elements, especially those 
obtained by {\em surgery}. In $3$-manifold topology, one is {\em a priori} interested in 
closed $3$-manifolds. But experience shows that $3$-manifolds obtained by (Dehn) surgery on
a fixed $3$-manifold with torus boundary are related in understandable ways. Similarly, even if
one is only interested in $\scl$ in free groups (for some of the reasons suggested in the
introduction), it is worthwhile to study how $\scl$ behaves under surgery on free products of
free Abelian groups of higher ranks. In this analogy, the free Abelian factors correspond to
the peripheral $\Z^2$ subgroups in the fundamental group of a $3$-manifold with torus boundary.

\begin{definition}
Let $\lbrace A_i \rbrace$ and $\lbrace B_i \rbrace$ be two families of free Abelian groups. A family of
homomorphisms $\rho_i: A_i \to B_i$ induces a homomorphism $\rho: *_i A_i \to *_i B_i$.
We say that $\rho$ is induced by {\em surgery}. If $C \in B_1^H(*_i A_i)$, then we say that
$\rho(C) \in B_1^H(*_i B_i)$ is {\em obtained by surgery} on $C$.
\end{definition}

By Corollary~\ref{injective_is_isometry}, it suffices to consider surgery in situations where each
$\rho_i$ is surjective after tensoring with $\R$.

One also studies {\em families} of surgeries, with fixed domain and range, in which the homomorphisms
$\rho_i$ depend linearly on a parameter.

\begin{definition}
With notation as above, let $\sigma_i:A_i \to B_i$ and $\tau_i:A_i \to B_i$ be two families of homomorphisms.
For each $p \in \Z$, define $\rho_i(p):A_i \to B_i$ by $\rho_i(p) = \sigma_i + p\tau_i$, and define
$\rho(p):*_i A_i \to *_i B_i$ similarly. We refer to the $\rho(p)$ as a {\em line} of surgeries. If
$C \in B_1^H(*_i A_i)$, we say the $\rho(p)(C)$ are {\em obtained by a line of surgeries} on $C$.
\end{definition}

\subsection{An example}\label{detailed_example_section}

In this section we work out an explicit example of a (multi-parameter) family of surgeries.
Given a $4$-tuple of integers $\alpha_1,\alpha_2,\beta_1,\beta_2$ we define
an element $w_{\alpha_1,\alpha_2,\beta_1,\beta_2}$ in $B_1^H(F_2)$ (or just $w_{\alpha,\beta}$ for
short) by the formula
$$w_{\alpha,\beta}: = a^{-\alpha_1 -\alpha_2} + b^{-\beta_1- \beta_2} +
a^{\alpha_1}b^{\beta_1} + a^{\alpha_2}b^{\beta_2}$$
We can think of this as a family of elements obtained by surgery on a fixed element in $B_1^H(\Z^2 * \Z^2)$.
We will derive an explicit formula for $\scl(w_{\alpha,\beta})$ in terms of $\alpha$ and $\beta$,
by the methods of \S~\ref{free_product_section}.

Note that by Example~\ref{power_isometry_example}, we can assume that $\alpha_1$ and $\alpha_2$
are coprime, and similarly for the $\beta_i$. We make this assumption in the sequel. 
Finally, after interchanging $a$ with $a^{-1}$ or $b$ with
$b^{-1}$ if necessary, we assume $\alpha_1$ and $\beta_1$ are strictly positive.
The calculation of $\scl(w)$ reduces to a finite number of cases. We concentrate on a
specific case; in the sequel we therefore assume:
$$\alpha_1 > \alpha_1+\alpha_2 > 0 > \alpha_2, \quad \beta_1 > \beta_1+\beta_2 > 0 > \beta_2$$

We write $F_2 = A*B$ where $A = \langle a\rangle$ and $B = \langle b \rangle$.
The set $T(A)$ has three elements, corresponding to the three substrings of $w$ of the
form $a^*$. We denote these elements $1,2,3$. Since $1$ is an Abelian loop, $T_2(A)$ has
five elements; i.e.\/ $T_2(A) = \lbrace (1,1), (2,2), (2,3), (3,2), (3,3) \rbrace$. Let
$v \in V_A$ have co-ordinates $v_1$ through $v_5$. By the definition of $V_A$, the $v_i$ are
non-negative. The constraint that $\partial(v) = 0$ is equivalent to $v_3=v_4$. 
Hence in the sequel we will equate $v_3$ and $v_4$, and write a vector in $V_A$ in the form
$(v_1,v_2,v_3,v_5)$. In this basis, the constraint that $h(v)=0$ reduces to
$$(\alpha_1 + \alpha_2)(v_1-v_3) = \alpha_1v_2 + \alpha_2v_5$$
which we rewrite as
$$v_3 = v_1 - (\alpha_1 v_2 + \alpha_2 v_5)/(\alpha_1 + \alpha_2)$$
See Figure~\ref{xyz_graph}. 

\begin{figure}[htpb]
\labellist
\small\hair 2pt
\pinlabel $v_1$ at 50 200
\pinlabel $v_2$ at 198 200
\pinlabel $v_3$ at 300 225
\pinlabel $v_5$ at 400 200
\pinlabel $v_3=v_4$ at 300 175
\endlabellist
\centering
\includegraphics[scale=0.75]{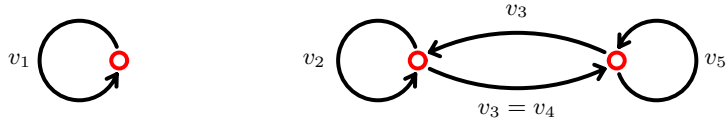}
\caption{The weighted graph $X(v)$ associated to $v\in V_A$, where
$v_3 = v_1 - (\alpha_1 v_2 + \alpha_2 v_5)/(\alpha_1 + \alpha_2)$ is necessarily non-negative.}\label{xyz_graph}
\end{figure}

The cone $V_A$ has four extremal vectors $\xi_i$ in $(v_1,v_2,v_3,v_5)$ co-ordinates,
which are the columns of the matrix
$$M_A: = \begin{pmatrix}
0 & 0 & 1 & \alpha_1 \\
-\alpha_2 & 0 & 0 & \alpha_1+\alpha_2 \\
0 & -\alpha_2 & 1 & 0 \\
\alpha_1 & \alpha_1 + \alpha_2 & 0 & 0 \\
\end{pmatrix}$$
By our assumption of the $\alpha_i$, the entries of this matrix are non-negative (as they must be).
Note that these four vectors are linearly dependent, and $V_A$ is the cone on a planar
quadrilateral. The cone $V_A$ is the image of the non-negative orthant in $\R^4$ under
multiplication on the left by $M_A$.

\medskip

A nonzero nonnegative vector $(v_1,v_2,v_3,v_5)$ is a disk vector if
and only if it is integral, if $v_1 = 0$, and if
$v_3 = - (\alpha_1 v_2 + \alpha_2 v_5)/(\alpha_1 + \alpha_2)> 0$.
Since we are assuming $\alpha_1+\alpha_2 > 0$, the disk vectors are all contained in
the face of $V_A$ spanned by $\xi_1$ and $\xi_2$ (i.e.\/ the face with $v_1=0$). The
set of disk vectors are precisely vectors of the form $(0,p,q,-(p+q)\alpha_1/\alpha_2-q)$
where $p,q$ are integers such that $q>0,p\ge 0$ and $-\alpha_2 | p+q$. Thus the 
Klein polyhedron $\conv(\D_A + V_A)$ has vertices $\xi_2$ and $d = (0,-1-\alpha_2,1,\alpha_1-1)$.

The Klein polyhedron has three faces which are not contained in faces of $V_A$. The
first face $K_1$ has vertex $\xi_2$, and has extremal rays $\xi_2 + t\xi_3$ and
$\xi_2 + t\xi_4$ for $t\ge 0$. The second face $K_2$ has vertices $\xi_2$ and $d$,
and has extremal rays $\xi_2+t\xi_4$ and $d+t\xi_4$, as well as the interval from $\xi_2$ to $d$.
The third face $K_3$ has vertex $d$ and extremal rays $d+t\xi_4$ and $d+t\xi_1$. The Klein
function $\kappa_A$ has the form
$$\kappa_A = \begin{cases}
\frac 1 {\alpha_2} \left( v_1 - \alpha_1 v_2/(\alpha_1+\alpha_2) - v_3 \right) &
\text{on the cone of }K_1 \\
\frac 1 {\alpha_2} \left( (\alpha_1+\alpha_2)v_1/\alpha_1 - v_2 - v_3 \right) &
\text{on the cone of }K_2 \\
v_3 & \text{on the cone of }K_3 \\
\end{cases}$$
while on all of $V_A$ we have $|v|/2 = (v_2+2v_3+v_5)/2$.

By the symmetry of $w_{\alpha,\beta}$, we obtain similar expressions for a typical vector
$(u_1,u_2,u_3,u_4,u_5)$ in $V_B$. 
With this notation, the polyhedron $Y$ is the subspace of $V_A \times V_B$ consisting
of vectors for which $u_2=v_2,u_3=v_3, u_4=v_4$ and $u_5=v_5$. The two Abelian loops themselves
impose no pairing conditions, but since we can write both $u_1$ and $v_1$ in terms of the other $u_i,v_i$
(in the same way), the equalities above imply $u_1=v_1$.

Setting $d(y)=[w_{\alpha,\beta}]$ imposes two more conditions on the vectors (at first glance it
looks like it imposes four conditions since there are four terms in $w$, 
but two of these conditions are already implicit in $v_3=v_4$ and $u_3=u_4$ which were consequences
of $\partial = 0$). These two extra conditions take the form
$v_2=v_5$ and $v_3 = 1-v_2$.

The two conditions give $v_1=1$, $v_2=v_5=x$ and $v_3=1-x$.
Making these substitutions, we find that $Y_{[w]}$ is the polygon $0 \le x \le 1$.

To compute $\scl$ we must maximize $\chi$ on $Y_{[w]}$. In terms of $x$, the function
$\chi$ is equal to $\kappa_A + \kappa_B - 2$ where
$$\kappa_A = \begin{cases}
x/(\alpha_1+\alpha_2) & \text{if } x \le (\alpha_1 +\alpha_2)/\alpha_1 \\
1/\alpha_1 & \text{if } (\alpha_1+\alpha_2)/\alpha_1 \le x \le (\alpha_1-1)/\alpha_1 \\
1-x & \text{if } x \ge (\alpha_1-1)/\alpha_1 \\
\end{cases}$$
and similarly for $\kappa_B$:
$$\kappa_B = \begin{cases}
x/(\beta_1+\beta_2) & \text{if } x \le (\beta_1 +\beta_2)/\beta_1 \\
1/\beta_1 & \text{if } (\beta_1+\beta_2)/\beta_1 \le x \le (\beta_1-1)/\beta_1 \\
1-x & \text{if } x \ge (\beta_1-1)/\beta_1 \\
\end{cases}$$
Then $\scl(w) = 1 - \max(\kappa_A(x) + \kappa_B(x))/2$.

\begin{proposition}\label{explicit_computation_proposition}
Let $w = a^{-\alpha_1-\alpha_2} + b^{-\beta_1-\beta_2} + a^{\alpha_1}b^{\beta_1} + a^{\alpha_2}b^{\beta_2}$
where the $\alpha_i$ and coprime and similarly for the $\beta_i$, and they satisfy
$$\alpha_1 > \alpha_1+\alpha_2 > 0 > \alpha_2, \quad \beta_1 > \beta_1+\beta_2 > 0 > \beta_2$$
We have the following formulae for $\scl(w)$ by cases:
\begin{enumerate}
\item{If $(\alpha_1-1)/\alpha_1 \le (\beta_1+\beta_2)/\beta_1$ then
$\scl(w) = 1 - \frac 1 2\left(\frac 1 {\alpha_1} + \frac {(\alpha_1-1)} {\alpha_1(\beta_1+\beta_2)}\right)$}
\item{If $(\beta_1-1)/\beta_1 \le (\alpha_1+\alpha_2)/\alpha_1$ then
$\scl(w) = 1 - \frac 1 2\left(\frac 1 {\beta_1} + \frac {(\beta_1-1)} {\beta_1(\alpha_1+\alpha_2)}\right)$}
\item{Otherwise $\scl(w) = 1 - \frac 1 2\left(\frac 1 {\alpha_1} + \frac 1 {\beta_1}\right)$}
\end{enumerate}
\end{proposition}

\begin{remark}
The program {\tt scallop} (see \cite{scallop}) implements an algorithm 
described in \cite{Calegari_scl} \S~4.1.7--8 to compute
$\scl$ on individual chains in $B_1^H(F_2)$, and can be used to give an independent check
of Proposition~\ref{explicit_computation_proposition}.
\end{remark}

Without much more work, we can also treat chains of the form
$$w'_{\alpha,\beta} = a^{-\alpha_1-\alpha_2} + b^{-\beta_1-\beta_2} + a^{\alpha_1}b^{\beta_1}a^{\alpha_2}b^{\beta_2}$$
The cones $V_A,V_B$ are the same but now the polyhedron $Y$ is slightly different,
defined by $u_2=v_3=u_4$ and $v_2=u_3=v_4$. Hence, in terms of the variable $x$,
the function $\kappa_A$ is as before, whereas $\kappa_B$ has the form:
$$\kappa_B = \begin{cases}
x & \text{if } x \le 1/\beta_1 \\
1/\beta_1 & \text{if } 1/\beta_1 \le x \le -\beta_2/\beta_1 \\
(1-x)/(\beta_1+\beta_2) & \text{if } x \ge -\beta_2/\beta_1 \\
\end{cases}$$

Hence we have
\begin{proposition}\label{explicit_computation_proposition_2}
Let $w' = a^{-\alpha_1-\alpha_2} + b^{-\beta_1-\beta_2} + a^{\alpha_1}b^{\beta_1}a^{\alpha_2}b^{\beta_2}$
where the $\alpha_i$ and coprime and similarly for the $\beta_i$, and they satisfy
$$\alpha_1 > \alpha_1+\alpha_2 > 0 > \alpha_2, \quad \beta_1 > \beta_1+\beta_2 > 0 > \beta_2$$
We have the following formulae for $\scl(w)$ by cases:
\begin{enumerate}
\item{If $-\beta_2/\beta_1\le (\alpha_1+\alpha_2)/\alpha_1$ then
$$\scl(w') = 1 - \frac 1 2\left(\max\left(\frac 1 {\beta_1} - \frac {\beta_2} {\beta_1(\alpha_1+\alpha_2)}, \frac 1 {\alpha_1}
- \frac {\alpha_2} {\alpha_1(\beta_1+\beta_2)}\right)\right)$$}
\item{Otherwise $\scl(w') = 1 - \frac 1 2\left(\frac 1 {\alpha_1} + \frac 1 {\beta_1}\right)$}
\end{enumerate}
\end{proposition}

The distribution of values of $\scl$ for all $w,w'$ with $\alpha_1,\beta_1\le 35$
(about $3$ million words) is illustrated in Figure~\ref{histogram_1_2}.

\begin{figure}[htpb]
\labellist
\small\hair 2pt
\pinlabel $0$ at 0 -15
\pinlabel $\frac{1}{2}$ at 300 -15
\pinlabel $1$ at 600 -15
\pinlabel $\frac{1}{3}$ at 200 -15
\pinlabel $\frac{2}{3}$ at 400 -15
\pinlabel $\frac{1}{4}$ at 150 -15
\pinlabel $\frac{3}{4}$ at 450 -15
\pinlabel $\frac{1}{5}$ at 120 -15
\pinlabel $\frac{2}{5}$ at 240 -15
\pinlabel $\frac{3}{5}$ at 360 -15
\pinlabel $\frac{4}{5}$ at 480 -15
\pinlabel $\frac{1}{6}$ at 100 -15
\pinlabel $\frac{5}{6}$ at 500 -15
\pinlabel $\frac{1}{7}$ at 85.7142857 -15
\pinlabel $\frac{2}{7}$ at 171.42857 -15
\pinlabel $\frac{3}{7}$ at 257.142857 -15
\pinlabel $\frac{4}{7}$ at 342.85714 -15
\pinlabel $\frac{5}{7}$ at 428.571428 -15
\pinlabel $\frac{6}{7}$ at 514.285714 -15
\endlabellist
\centering
\includegraphics[scale=0.5]{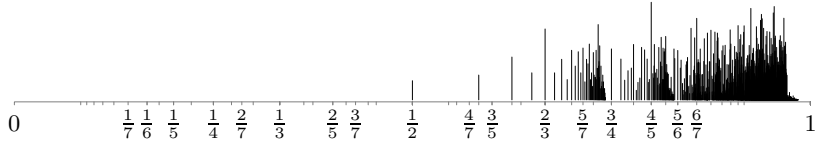}
\caption{Histogram of values of $\scl$ on $w,w'$}\label{histogram_1_2}
\end{figure}

This statement about integral chains in $F_2$ can be translated into a statement
about elements of the commutator subgroup of $F_4$. 

Let $a,b,c,d$ be generators for a free group $F_4$. For each $\alpha,\beta$
define
$$w''_{\alpha,\beta} = a^{-\alpha_1-\alpha_2}cb^{-\beta_1-\beta_2}c^{-1}da^{\alpha_1}b^{\beta_1}a^{\alpha_2}b^{\beta_2}d^{-1}$$
By the self-product formula, i.e.\/ Theorem~2.101 from \cite{Calegari_scl} (also see Remark~2.102), 
we have an equality
$$\scl_{F_4}(w''_{\alpha,\beta}) = \scl_{F_2}(w'_{\alpha,\beta}) + 1$$ 
Consequently, by the multiplicativity of $\scl$ under taking powers, we deduce the following theorem:

\begin{theorem}[Denominators]\label{denominator_theorem}
The image of a nonabelian free group of rank at least $4$ under $\scl$ in
$\R/\Z$ is precisely $\Q/\Z$.
\end{theorem}

If $F_1$, $F_2$ are free groups containing elements $w_1$ and $w_2$ respectively, then 
by the free product formula, i.e.\/ Theorem~2.93 from \cite{Calegari_scl}, we have an equality
$\scl(w_1w_2) = \scl(w_1) + \scl(w_2) + 1/2$ where the product on the left hand side is taken in the
free group $F_1 * F_2$. Suppose $w_i$ is an infinite family of elements in $F$ for which the set of 
numbers $\scl(w_i)$ is well-ordered with ordinal type $\omega$. Then we can take two copies 
$F_1,F_2$ of $F$, and corresponding elements $w_{i,1},w_{i,2}$
in each copy, and observe that the set of numbers $\scl(w_{i,1}w_{j,2})$ is well-ordered with ordinal 
type $\omega^2$. Repeating this process inductively, we deduce the following theorem:

\begin{theorem}[Limit values]
For each $n$, the image of the free group $F_n$ under $\scl$ contains a well-ordered sequence of values
with ordinal type $\omega^{\lfloor n/4 \rfloor}$. 
The image of $F_\infty$ under $\scl$ contains a well-ordered sequence of values with 
ordinal type $\omega^\omega$.
\end{theorem}

To obtain stronger results, it is necessary to understand how $\chi$ varies as a function of
the parameters in a more general surgery family. 

\subsection{Faces and signatures}

We recall the method to compute $\scl(w)$ described in \S~\ref{free_product_section} for
a chain $w \in B_1^H(A*B)$. In broad outline, the method has three steps:
\begin{enumerate}
\item{Construct the polyhedra $V_A$ and $V_B$ and $Y_w \subset V_A \times V_B$}
\item{Express $\chi$ as the minimum of a finite family of rational linear functions}
\item{Maximize $\chi$ on $Y_w$}
\end{enumerate}
In principle, step (1) is elementary linear algebra. However in practice, even for
simple $w$ the polyhedra $V_A,V_B,Y_w$ become difficult to work with directly,
and it is useful to have a description of these polyhedra which is as simple
as possible; we take this up in \S~\ref{combinatorics_subsection}.

Given $Y_w$ and $\chi$, step (3) is a straightforward linear programming problem, which
may be solved by any number of standard methods (e.g.\/ Dantzig's simplex
method \cite{Dantzig}, Karmarkar's projective method \cite{Karmarkar} and so on). These
methods are generally very rapid and practical.

The ``answers'' to steps (1) and (3) depend piecewise rationally linearly on the parameters of
the problem, and it is easy to see their contribution to $\scl$ on families obtained by a line
of surgeries.

The most difficult step, and the
most interesting, is (2): obtaining an explicit description of $\chi$ as a function of a
parameter $p$ in a line of surgeries. Because of Proposition~\ref{norm_is_linear}, this
amounts to the determination of the respective Klein functions $\kappa$ 
on each of $V_A$ and $V_B$. This turns out to be a very difficult question to answer precisely,
but we are able to obtain some qualitative results.

\medskip

For a given combinatorial type of $V_A$, it takes a finite amount of data to specify the set
of open faces with connected support (i.e.\/ those faces with the property that the integer lattice
points they contain are disk vectors). We call this data that {\em signature} of $V_A$, and
denote in $\sign(V_A)$. Evidently the sail of $V_A$ depends only on $\sign(V_A)$ (a finite amount
of data), and the orbit of $V_A$ under $\GL(C_2(A),\Z)$.

\subsection{Combinatorics of $V_A$}\label{combinatorics_subsection}

In this section we will give an explicit description of 
$V_A$ as a polyhedron depending on $w$. Recall
that $V_A$ is the set of non-negative vectors in $C_2(A)$ in the kernel
of both $\partial$ and $h$. 
Define $W_A$ to be the set of non-negative vectors in $C_2(A)$ in the kernel
of $\partial$. Hence $V_A = W_A \cap \ker(h)$.
We first give an explicit description of $W_A$.

\smallskip

Let $\Sigma$ denote the directed graph with vertex set $T(A)$ and edge set
$T_2(A)$.
Non-negative vectors in $C_2(A)$ correspond to simplicial $1$-chains,
whose simplices are all oriented compatibly with the orientation on the
edges of $\Sigma$. A vector is in the kernel of $\partial$ if and only if
the corresponding chain is a $1$-cycle. Hence we can think of $W_A$
as a rational convex polyhedral cone in the real vector space $H_1(\Sigma)$.

A $1$-cycle in $H_1(\Sigma)$ is determined by the degree with which
it maps over every oriented edge of $\Sigma$. A $1$-cycle $\phi$ in $W_A$
determines an oriented subgraph $\Sigma(\phi)$ of $\Sigma$ which is the union
of edges over which it maps with strictly positive degree.

\begin{lemma}\label{faces_recurrent}
An oriented subgraph of $\Sigma$ is of the form $\Sigma(\phi)$ for
some $\phi \in W_A$ if and only if every component is {\em recurrent}; i.e.\/ it contains
an oriented path from every vertex to every other vertex.
\end{lemma}
\begin{proof}
For simplicity restrict attention to one component.
A connected oriented graph is recurrent if and if it contains no {\em dead ends}: i.e.\/
partitions of the vertices of $\Sigma$ into nonempty subsets $Z_1,Z_2$
such that every edge from $Z_1$ to $Z_2$ is oriented positively. Since
$\phi$ is a cycle, the flux through every vertex is zero. If there were a
dead end $Z_1,Z_2$ the flux through $Z_2$ would be positive, which is absurd.
Hence $\Sigma(\phi)$ is recurrent.

Conversely, suppose $\Gamma$ is recurrent. For each oriented edge $e$ in $\Gamma$,
choose an oriented path from the endpoint to the initial point of $e$ and concatenate
it with $e$ to make an oriented loop. 
The sum of these oriented loops is a $1$-cycle $\phi$ for which
$\Sigma(\phi) = \Gamma$.
\end{proof}

Lemma~\ref{faces_recurrent} implies that the faces of $W_A$ are in bijection with
the recurrent subgraphs $\Gamma$ of $\Sigma$. The dimension of the face corresponding
to a graph $\Gamma$ is $\dim(H_1(\Gamma))$.
As a special case, we obtain the following:

\begin{lemma}\label{rays_loops}
The extremal rays of $W_A$ are in bijection with oriented embedded loops
in $\Sigma$.
\end{lemma}

\begin{example}
Given a graph $\Gamma$ (directed or not), there is a natural graph $O(\Gamma)$ whose vertices are
embedded oriented loops in $\Gamma$, and whose edges are pairs of oriented loops whose 
union has $\dim(H_1)=2$. In the case that $\Gamma$ is the $1$-skeleton of a tetrahedron, 
the graph $O(\Gamma)$ is the $1$-skeleton of a stellated cube; see Figure~\ref{chambers}.

\begin{figure}[htpb]
\labellist
\small\hair 2pt
\endlabellist
\centering
\includegraphics[scale=0.5]{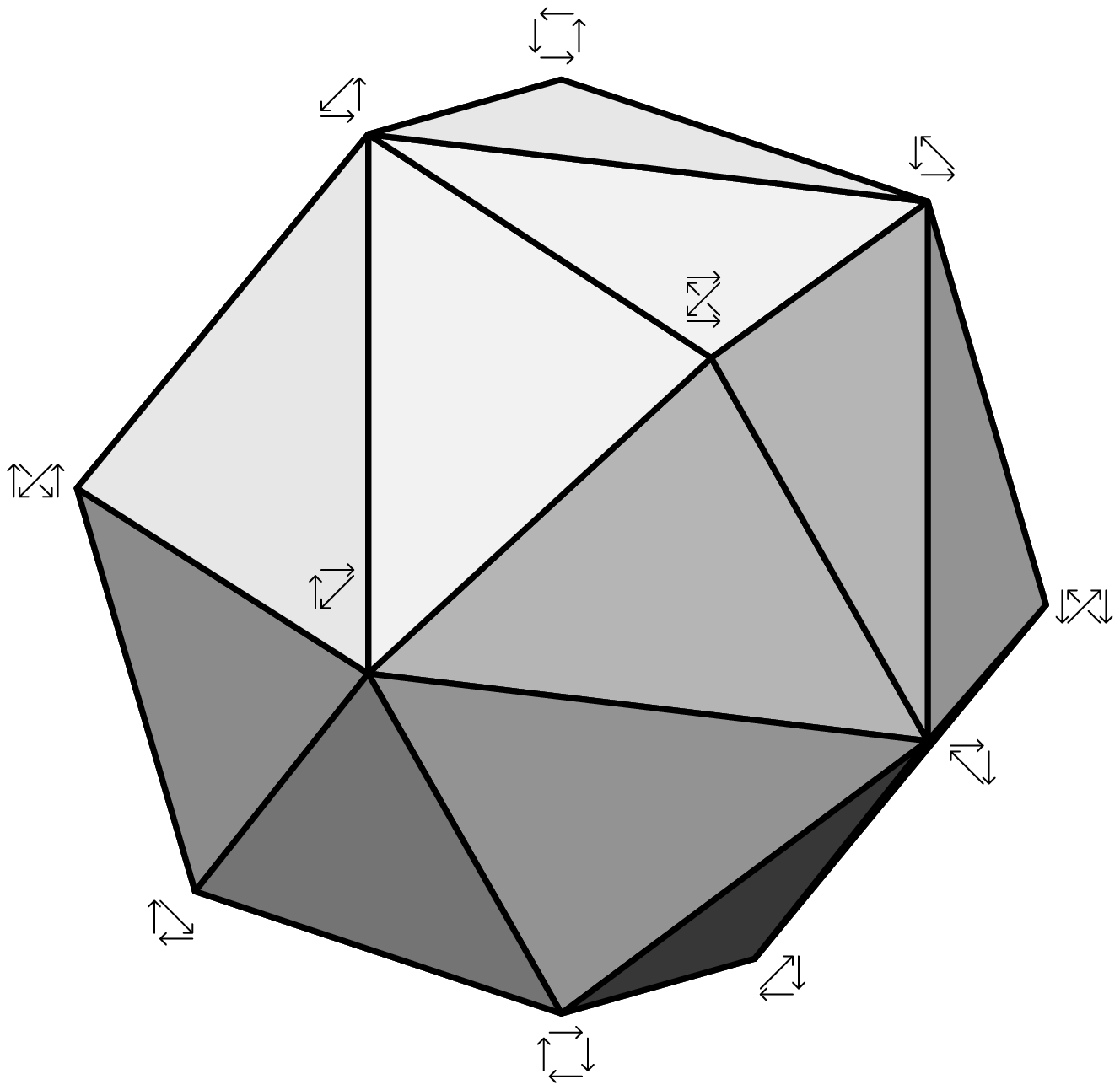}
\caption{}\label{chambers}
\end{figure}
\end{example}

The polyhedron $W_A$ depends only weakly on the precise form of $w$.
In fact, discounting Abelian loops, the polyhedron $W_A$ only depends on the cardinality of
$T(A)$. To describe $V_A$ we need to consider the function $h:C_2(A) \to A$. Recall that
$h(\tau,\tau')=(\tau+\tau')/2$ where we identify elements of $T(A)$ with elements
of $A$ by thinking of the $T(A)$ as loops in a $K(A,1)$. Denote this
identification map by $i:T(A) \to A$ and think of $i$ as a function on the 
vertices of $\Sigma$, so that if $\phi$ is an embedded loop in $\Sigma$, then
$h(\phi) = \sum_{v \in \phi}i(v) \in A$. Since by hypothesis $w \in B_1^H(F)$, we have
$h(\phi)=0$ whenever $\phi$
is a Hamiltonian circuit (an embedded loop which passes through each vertex exactly
once). Moreover for {\em generic} $w \in B_1^H(F)$ and generic $i$,
these are the {\em only} embedded loops with $h=0$.
In any case, we obtain a concrete description of $V_A$, or, equivalently,
of the set of extremal rays.

\begin{lemma}\label{rays_of_V}
Rays of $V_A$ are in the projective class of two kinds of $1$-cycles:
\begin{enumerate}
\item{embedded oriented loops $\phi$ in $\Sigma$ with $h(\phi)=0$ 
(which includes the Hamiltonian circuits in $\Sigma$)}
\item{those of the form $h(\phi')\phi - h(\phi)\phi'$ where
$\phi,\phi'$ are distinct embedded oriented loops whose intersection is connected (and possibly
empty), with $h(\phi')>0$ and $h(\phi)<0$}
\end{enumerate}
\end{lemma}
\begin{proof}
The rays of $V_A$ are the intersection of the hyperplane $h=0$ with the rays
and $2$-dimensional faces of $W_A$. The rays of $W_A$ which satisfy $h=0$ are exactly
the projective classes of the oriented loops $\phi$ with $h(\phi)=0$. 
The $2$-dimensional faces of $W_A$
correspond to the recurrent subgraphs $\Gamma$ with
$\dim(H_1(\Gamma))=2$. By Mayer-Vietoris, such a $\Gamma$ is the union of a
pair of embedded loops $\phi,\phi'$
whose (possibly empty) intersection is connected. The hyperplane $h=0$ intersects such a face in
a ray in the projective class of $h(\phi')\phi - h(\phi)\phi'$.
\end{proof}

\subsection{Surgery theorem}\label{surgery_subsection}

In what follows, we fix $A*B$ and a linear family of surjective homomorphisms (i.e.\/ a line of surgeries)
$\rho_p:A*B \to A' * B$ where $A,A',B$ are free
Abelian, and $\rank(A') = \rank(A)-1$. Fix $w \in B_1^H(A*B)$ and denote $w(p):=\rho_p(w)$.

\medskip

Recall that the set of disk vectors $\D_A$ in $V_A$ is the union of the integer lattice points in those
open faces $F$ of $V_A$ for which $\supp(F)$ is connected. In a line of surgeries, the polyhedra
$V_A(p)$ vary in easily understood ways. For each $p$, let $M(p)$ be an integral matrix whose columns are 
vectors spanning the extremal rays of $V_A(p)$. Then $M(p)$ has the form
$M(p) = N + pN'$, where $N$ and $N'$ are fixed integral matrices, depending only on $w$. 
As $p \to \infty$, the cones $V_A(p)$ converge
in the Hausdorff topology to a rational cone $V_A'(\infty)$ spanned by the nonzero columns of $N'$,
and the columns of $N$ corresponding to the zero columns of $N'$. The cone $V_A(\infty)$ associated
to $w$ has codimension one in each $V_A(p)$, and codimension one in the limit $V_A'(\infty)$.

For each $p$, let $\D_A(p)$ denote the disk vectors in $V_A(p)$, and $\D_A(\infty)$ the disk
vectors in $V_A(\infty)$. Similarly, let $\kappa_p$ denote the Klein function on $V_A(p)$, and
$\kappa_\infty$ the Klein function on $V_A(\infty)$. Observe that any $v$ that is in $\D_A(p)$ 
for some $p$ is also in $\D_A(q)$ for all $q$ such that $v$ is in $V_A(q)$ 
(i.e.\/ the property of being a disk vector does not depend on $p$).

\begin{lemma}\label{disks_in_subspace_lemma}
There is convergence in the Hausdorff topology
$$\conv(\D_A(p)+V_A(p)) \to \conv(\D_A(\infty) + V_A'(\infty))$$
Hence $\kappa_\infty = \lim_{p \to \infty} \kappa_p |_{V_A(\infty)}$.
\end{lemma}
\begin{proof}
The set of integer lattice points is discrete. Since every integer lattice point is either in every
$V_A(p)$ or in only finitely many, the intersection of $\D_A(p)$ with any compact subset of $W_A$
is eventually equal to the intersection of this compact set with $\cap_p \D_A(p)$.
Since $V_A(\infty) = \cap_p V_A(p)$, we have $\D_A(\infty) = \cap_p \D_A(p)$.

The last claim follows because 
$$\conv(\D_A(\infty)+V_A'(\infty)) \cap V_A(\infty) = \conv(\D_A(\infty) + V_A(\infty))$$
\end{proof}

From this discussion we derive the following theorem.

\begin{theorem}[Surgery]\label{surgery_limit_theorem}
Fix $w \in B_1^H(A*B)$ and let $\rho_p:A*B \to A'*B$ be a line of surgeries, constant on the
second factor, and surjective on the first factor with $\rank(A') = \rank(A)-1$.
Define $w(p)=\rho_p(w)$.
Then $\lim_{p \to \infty} \scl(w(p)) = \scl(w)$.
\end{theorem}
\begin{proof}
Denote $\lim_{p \to \infty} \kappa_p=\kappa_\infty'$, thought of
as a function on $V_A'(\infty)$. Denote by
$Y_w' \subset V_A'(\infty) \times V_B$ the limit as $p \to \infty$ of
$Y_w(p) \subset V_A(p) \times V_B$. Since $|v_A|$ is constant on each $Y_w(p)$,
it follows that $|v_A|$ is also constant on $Y_w'$. Lemma~\ref{disks_in_subspace_lemma}
implies that the only disk vectors in $V_A'(\infty)$ which contribute to
$\kappa_\infty'$ are those in $V_A(\infty)$.
Since $|v_A|$ is non-negative on $V_A'(\infty)$, a level set of $|v_A|$ which is a supporting
hyperplane for $\conv(\D_A(\infty)+V_A(\infty))$ is also a supporting hyperplane
for $\conv(\D_A(\infty)+V_A'(\infty))$. It follows that $\chi$ restricted
to $Y_w'$ is maximized in $Y_w$. The result follows by applying
Theorem~\ref{rationality_theorem}.
\end{proof}

\begin{remark}
By monotonicity of $\scl$ under homomorphisms one has the inequality $\scl(w(p)) \le \scl(w)$ for all $p$.
Thus surgery ``explains'' the existence of many nontrivial accumulation points in the $\scl$ spectrum of
a free group. However it should also be pointed out that equality is sometimes achieved in families,
so that $\scl(w(p))=\scl(w)$ for all $p$ (for example, under the conditions discussed
in Example~\ref{monotone_words}).
\end{remark}

It is interesting to note that the limit does {\em not} depend on the particular
surgery family.

\begin{example}
Let $w=a^2c^2ba^{-1}b^{-1}c^{-1}ba^{-1}c^{-1}b^{-1}$ where $[a,c]=\id$, and $b$ generates a free summand. 
Consider the line of surgeries defined by $\rho_p(a)=a$, $\rho_p(b)=b$ and $\rho_p(c)=pa$. In this case, 
$w(p)=a^{2+2p}ba^{-1}b^{-1}a^{-p}ba^{-1-p}b^{-1}$. Then
$$\scl(w(p))=\frac {4p+3}{4p+4} \text{ if }p\text{ is odd, and }\frac {2p+1}{2p+2}\text{ if }p\text{ is even.}$$
Define $\sigma =\bigl( \begin{smallmatrix} 1&0\\ 1&1 \end{smallmatrix} \bigr)$ and consider the line of surgeries 
obtained by the same homomorphisms $\rho_p$ precomposed with the automorphism $\sigma$ of $\Z^2$. Then
$$\scl(w(p)^\sigma)=\frac {8p+3}{8p+4} \text{ if }p\text{ is odd, and }\frac {4p+5}{4p+6}\text{ if }p\text{ is even.}$$
Both sequences of numbers converge to $1=\scl(w)$ as $p \to \infty$.
Note that even when the values of $\scl(w(p))$ and $\scl(w(q)^\sigma)$ agree, the corresponding elements are
typically not in the same $\Aut$ orbit in $F_2$.
\end{example}

\begin{remark}
The precise algebraic form of $\scl(w(p))$ on surgery families is analyzed in a
forthcoming paper of Calegari-Walker \cite{Calegari_Walker},
where it is shown quite generally that $\scl(w(p))$ is a {\em ratio of quasipolynomials} in $p$,
for $p\gg 0$.
\end{remark}

\subsection{Computer implementation}

The algorithm described in this paper has been implemented by Alden Walker in the program
{\tt sss}, available from the author's website \cite{Walker}. This allows computations that would
be infeasible with {\tt scallop}; e.g.\/ $\scl(aba^{-98}ba^{-1}b^{-3}+a^{98}b) = 195/196$.
The runtime in this implementation appears to
be doubly exponential in the number of $A$ and $B$ arcs, and the practical limit for this number
appears to be no more than about $5$ (excluding Abelian loops).

Some theoretical explanation for this
difficulty comes from recent work of Lukas Brantner \cite{Brantner}, who shows (amongst
other things) that the
problem of deciding whether a given disk vector $d\in \D_A$ is {\em essential} --- i.e.\/
cannot be written as $d=e+v$ for $e\in \D_A$ and $v\in V_A-0$ --- is $\textnormal{\sf coNP}$-complete.

\section{Acknowledgment}

I would like to thank Lukas Brantner, Jon McCammond, Alden Walker and the referee
for helpful comments and corrections. While writing this paper I was partially supported 
by NSF grants DMS 0707130 and DMS 1005246.


\begin{thebibliography}{99}
\bibitem{Arnold}
	V. Arnold,
	\emph{Higher-dimensional continued fractions},
	Regul. Chaotic Dyn. {\bf 3} (1998), no. 3, 10--17
\bibitem{Barvinok}
	A. Barvinok,
	\emph{Integer Points in Polyhedra},
	Zurich lectures in advanced mathematics. EMS, Zurich, 2008
\bibitem{Bavard}
	C. Bavard,
	\emph{Longueur stable des commutateurs},
	Enseign. Math. (2), {\bf 37}, 1-2, (1991), 109--150
\bibitem{Bollobas}
	B. Bollob\'as,
	\emph{Modern graph theory},
	Springer GTM {\bf 184}, Springer-Verlag, New York, 1998
\bibitem{Brantner}
	L. Brantner,
	\emph{On the complexity of sails},
	preprint, in preparation
\bibitem{Burger_Monod}
	M. Burger and N. Monod,
	\emph{On and around the bounded cohomology of $\SL_2$},
	Rigidity in dynamics and geometry (Cambridge, 2000), Springer, Berlin, 2002, 19--37
\bibitem{Calegari_pickle}
	D. Calegari,
	\emph{Stable commutator length is rational in free groups},
	Jour. AMS, {\bf 22} (2009), no. 4, 941--961
\bibitem{Calegari_scl}
	D. Calegari,
	\emph{scl},
	MSJ Memoirs, {\bf 20}. Mathematical Society of Japan, Tokyo, 2009
\bibitem{Calegari_Walker}
	D. Calegari and A. Walker,
	\emph{Integer hulls of linear polyhedra and scl in families},
	preprint, in preparation
\bibitem{scallop}
	D. Calegari and A. Walker,
	{\tt scallop},
	computer program, available from the authors' websites
\bibitem{Dantzig}
	G. Dantzig,
	\emph{Linear Programming and Extensions},
	Princeton Univ. Press, Princeton, 1963
\bibitem{Gabai}
	D. Gabai,
	\emph{Foliations and the topology of $3$-manifolds},
	J. Diff. Geom. {\bf 18} (1983), no. 3, 445--503
\bibitem{Gordan}
	P. Gordan,
	\emph{\"Uber die Aufl\"osung linearer Gleichungen mit reelen Coefficienten},
	Math. Ann. {\bf 6}, (1873), 23--28
\bibitem{Gromov_bounded}
	M. Gromov,
	\emph{Volume and bounded cohomology},
	Inst. Hautes \'Etudes Sci. Publ. Math. No. 56 (1982), 5--99
\bibitem{Karmarkar}
	N. Karmarkar,
	\emph{A new polynomial-time algorithm for linear programming},
	Combinatorica {\bf 4} (1984), no. 4, 373--395
\bibitem{Klein}
	F. Klein,
	\emph{\"Uber eine geometrische Auffassung der gew\"ohnlichen
	Kettenbruchentwicklung}, Nachr. Ges. Wiss. G\"ottingen, Math.-Phys. {\bf 3},
	(1895), 357--359
\bibitem{Maclane}
	S. Mac Lane,
	\emph{Homology},
	Springer classics in mathematics, Berlin, 1995
\bibitem{Stallings}
	J. Stallings,
	\emph{Surfaces mapping to wedges of spaces. A topological variant of
	the Grushko-Neumann theorem},
	Groups---Korea '98 (Pusan), 345--353, de Gruyter, Berlin, 2000
\bibitem{Walker}
	A. Walker,
	{\tt sss},
	computer program, available from the author's website
\bibitem{Zhuang}
	D. Zhuang,
	\emph{Irrational stable commutator length in finitely presented groups},
	J. Mod. Dyn. {\bf 2} (2008), no. 3, 499--507
\end{thebibliography}
\end{document}